\documentclass[12pt,a4paper]{amsart}

\usepackage{amsmath}
\usepackage{amssymb}

\theoremstyle{plain}   
\begingroup 
\newtheorem*{Proposition S}{Proposition S}

\newtheorem{theorem}{Theorem}[section]   
\newtheorem{corollary}[theorem]{Corollary}     
\newtheorem{lemma}[theorem]{Lemma}         
\newtheorem{proposition}[theorem]{Proposition}  
\endgroup

\theoremstyle{definition}
\newtheorem{definition}[theorem]{Definition}   

\theoremstyle{remark}
\newtheorem{remark}[theorem]{Remark}        

\numberwithin{equation}{section}

\newcommand{\ep}{\varepsilon}

\newcommand{\R}{{\mathbb R}}
\newcommand{\N}{{\mathbb N}}

\newcommand{\cal}{\mathcal}

\newcommand{\lin}{\operatorname{span}}

\newcommand{\vf}{\varphi}

\newcommand{\dist}{\operatorname{dist}}

\newcommand{\Li}{{\mathcal{L}}}
\newcommand{\Ki}{{\mathcal{K}}}

\newcommand{\F}{{\cal F}}

\usepackage[dvips]{color}

\begin{document}

\title[Fr\' echet differentiability via partial Fr\' echet differentiability]{Fr\' echet differentiability via partial Fr\' echet differentiability}

\author{Lud\v ek Zaj\'\i\v{c}ek}


\email{zajicek@karlin.mff.cuni.cz}

\address{Charles University,
Faculty of Mathematics and Physics,
Sokolovsk\'a 83,
186 75 Praha 8-Karl\'\i n,
Czech Republic}

\bigskip

\begin{abstract}
Let $X_1, \dots, X_n$ be Banach spaces and $f$ a real function on $X=X_1 \times\dots \times X_n$. Let
 $A_f$ be the set of all points $x \in X$ at which $f$ is partially Fr\' echet
 differentiable
but is not Fr\' echet
 differentiable. Our  results imply that if  $X_1, \dots, X_{n-1}$ are Asplund spaces and  $f$ is continuous
 (resp. Lipschitz) on $X$, then $A_f$ is a first category  set (resp. a $\sigma$-upper porous set).
We also prove that if $X$, $Y$ are separable Banach spaces and $f:X \to Y$ is a Lipschitz mapping, then the set
 of all points $x \in X$ at which  $f$ is  G\^ ateaux differentiable, is  Fr\' echet differentiable
 along a closed subspace of finite codimension but is not Fr\' echet
 differentiable, is $\sigma$-upper porous. A number of related more general results are also proved.
\end{abstract}


\markboth{L.~Zaj\'{\i}\v{c}ek}{Fr\' echet differentiability via partial Fr\' echet differentiability}

\maketitle


\smallskip


\section{Introduction}\label{Intr}
If $f$ is a real function on $\R^n$, denote by $A_f$ the set of all points $x \in \R^n$ at which $f$
 has all (finite) partial derivatives  $f_1'(x),\dots,f_n'(x)$ but it is not (Fr\' echet) differentiable. Of course, $A_f$
 can be nonempty and Stepanoff's \cite{St} examples show that $A_f$ can have positive measure
  for a continuous function on $\R^2$ (he constructed  even such a function which is everywhere partially differentiable and also a continuous function on $\R^2$ which is partially differentiable almost everywhere
	but it is nowhere differentiable).
	
	However, the situation is different, if we consider ``topological smallness'' instead of ``measure smallness'',
	 as the following result shows.
	\smallskip
	
	{\bf Theorem  C}\ \ Let $G \subset \R^n$ be an open set and $f: G \to \R$ a continuous function. Then
	$$A_f:= \{x \in G:\ f_1'(x) \in \R,\dots, f_n'(x) \in \R\ \text{and}\ f'(x)\ \text{does not exist}\}$$
	 is a first category (=meagre) set.
	\smallskip
	
	For $n=2$, this result follows from Gorlenko's 1977 article \cite{G}. For general $n$, it was proved
	 (even for $f:G \to Y$, where $Y$ is a separable Banach space) by Ka-Sing-Lau and Weil \cite{LW}
	 in 1978.
	
	A related remarkable result from \cite{Sa} gives that the conclusion of Theorem C holds if $f$ is {\it an arbitrary}
	 function which is partially differentiable everywhere in $G$.
	\smallskip
	
	Bessis and Clarke \cite{BC} in 1999 proved the following ``Lipschitz version'' of Theorem C.
	\smallskip
	
	{\bf Theorem  L}\ \ Let $G \subset \R^n$ be an open set and $f: G \to \R$ a Lipschitz function. Then
	$$A_f:= \{x \in G:\ f_1'(x) \in \R,\dots, f_n'(x) \in \R\ \text{and}\ f'(x)\ \text{does not exist}\}$$
	 is a $\sigma$-porous set.
	
	\smallskip
	
	In the present article,  $\sigma$-porosity is ``$\sigma$-upper porosity'' (see Definition \ref{dpor} below, cf. \cite{Za05}), i.e. it is considered in
	 ``Denjoy-Dolzhenko sense''. Note that if $A\subset \R^n$ is $\sigma$-porous, then it is both of the first category and Lebesgue null, but the opposite implication does not hold. So Theorem L does not follow from
	  Theorem C and the Rademacher theorem.
	
	In the present article we prove some generalizations of Theorem C and Theorem L in the infinite-dimensional
	 setting. Namely, let $X_1,\dots,X_n$ and $Y$ be Banach spaces, $G$ an open subset of $X:= X_1 \times \cdots \times X_n$ (equipped with the maximum norm), and $f:G \to Y$ a mapping. Denote by $A_f$ the set of all
	points $x \in G$ at which all partial Fr\' echet derivatives  of $f$ exist but the Fr\' echet derivative $f'(x)$
	 does not exist. 
	
	In Section  \ref{cont} we  prove (in Theorem \ref{sousp})
	that $A_f$ is a first category set whenever $f$ is continuous and
	\begin{equation}\label{jsousepa}
	\text{the spaces of continuous linear mappings}\ \Li(X_1,Y),\dots,\Li(X_{n-1},Y)\ \text{are separable}.
	\end{equation}
	
	We obtain this result as an immediate consequence of an easy known result from \cite{Pe} (Proposition
	 \ref{penot} below) and  Theorem \ref{dirfrstr} (which can be of an independent interest) which
	 says that, under some conditions, Fr\' echet differentiability along a subspace $V$ generically 
	 implies strict differentiability along $V$.
		
	In Section  \ref{lip} we prove by a different (more technical) method Theorem \ref{soulip} which implies that, under condition \eqref{jsousepa}, the generalization
	 of Theorem L also holds.
	
	In Section \ref{sepred}, using the method of separable reduction, we show in Theorem \ref{sousepred}
	that generalizations of Theorem C and
	 Theorem L hold also under a condition more general than  condition \eqref{jsousepa}. 
	 In particular, we prove that  generalizations of Theorem C and
	 Theorem L hold if $Y=\R$ and  $X_1,\dots, X_{n-1}$ are Asplund spaces  (Corollary \ref{aspl}).
	
	I do not believe that condition  \eqref{jsousepa} can be omitted in Theorem \ref{sousp} and/or
	 Theorem \ref{soulip}. Unfortunately, I was not able to find any counterexample. So it is still
	 possible  that the Banach space generalizations of Theorem C and/or Theorem L hold in the full
	 generality and from this reason I do not discuss all cases in which the validity of these generalizations
	 follow by the  methods of the present article (cf. Remark \ref{soured}).
	
	In the proof of Theorem \ref{sousepred},
  we use a result (Proposition \ref{borpart}) on the Borel type of the set of all points at which a partial derivative of a continuous (or a slightly more general) function
	exists. This result 
	 which generalizes a proposition on functions on $\R^2$ from \cite{MP} and can be of an independent interest 
	is proved in Section 
	\ref{borel}. 
	
	We consider (mainly in Section \ref{lip}) also related problems where instead of a product space  $X:= X_1 \times \cdots \times X_n$
	 we consider an arbitrary Banach space $X$ and instead of Fr\' echet partial derivatives  we consider
	 Fr\' echet derivatives along subspaces.  In this direction, we obtain Propositions \ref{dirsp}, Proposition \ref{dirlip}, and Corollary \ref{fregat} which is an immediate consequence of more general
	 Proposition \ref{kkodsx}, which can be of an independent interest. An another consequence of this proposition is   Corollary \ref{gat} which says
	 that if $X$, $Y$ are separable Banach spaces and $f:X \to Y$ is a Lipschitz mapping, then the set
 of all points $x \in X$ at which  $f$ is  G\^ ateaux differentiable, is  Fr\' echet differentiable
 along a closed subspace of  finite codimension but is not Fr\' echet
 differentiable, is $\sigma$-porous.

\section{Preliminaries}\label{Prel}

 In the following, we consider  real notrivial (i.e. not equal to  $\{0\}$) Banach spaces. In any fixed Banach space,  we denote the zero vector by
$0$ and the norm by $|\cdot|$.  By a {\it subspace} $Y$ of a Banach space $X$, we will
  mean a Banach subspace of $X$, i.e. a closed linear subspace $Y \neq \{0\}$. 
	We set $S_X:= \{x \in X: |x|=1\}$.
  By $\lin M$ we denote the linear span of $M \subset X$. 
   The equality $X= X_1\oplus\cdots\oplus X_n$ means that the Banach space 
	$X$ is the topological direct sum of its subspaces $X_1,\dots,X_n$. 
	The symbol $B(x,r)$ will denote the open ball
  with center $x$ and radius $r$. 
\begin{definition}\label{dpor}
Let $A$ be a subset of a Banach space $X$.
\begin{enumerate}
\item[(i)] We say that $A$ is porous at a point $x \in X$ if there exists $c>0$ such that for each $\delta>0$
 there exists $t \in B(x, \delta)\setminus \{x\}$ such that $B(t, c |t-x|) \cap A = \emptyset$.
\item[(ii)] 
$A$ is called a porous set if $A$ is porous at each point $x \in A$.
\item[(iii)] 
$A$ is called a $\sigma$-porous set if it is a countable union of porous sets.
\end{enumerate}
\end{definition}
	
	If $X$ and $Y$ are Banach spaces, the space of all continuous linear mappings
	 $\vf: X \to Y$ (equipped with the usual norm) will be denoted by $\cal L(X,Y)$.

	The word {\it ``generically''} has the usual sense; it means ``at all points except a first category set''.
  
  Recall that a Banach space $X$ is called an Asplund space if each continuous convex function on $X$ is generically Fr\' echet differentiable and that
  \begin{equation}\label{asp}
  \text{$X$ is Asplund if and only if $Y^*$ is separable for each separable subspace $Y \subset X$.}
  \end{equation} 
	
	Let $X$, $Y$ be Banach spaces, $G \subset X$ an open set, and $f:G \to Y$ a mapping. 
    We say  that {\it $f$ is Lipschitz at $x \in G$} if  $\limsup_{y \to x} \frac{|f(y)-f(x)|}{|y-x|} < \infty$.
   The directional and one-sided  directional derivatives 
    of $f$ at $x\in G$ in the direction $v\in X$ are defined respectively by
   $$f'(x,v) := \lim_{t \to 0} \frac{f(x+tv)-f(x)}{t}\ \ \text{and}\ \  f'_+(x,v) := \lim_{t \to 0+} \frac{f(x+tv)-f(x)}{t}.$$

\begin{definition}
Let  $X$ and $Y$ be Banach spaces,  $ V$ a closed subspace of $X$, $G \subset X$
 an open set, $a \in G$ and $f: G \to Y$ a mapping. We say that $f$ is {\it Fr\' echet differentiable  at  $a$
 along $V$}, if the mapping $g(v):= f(a+v),\ v \in V \cap (G-a)$, is
 Fr\' echet differentiable at $0 \in V$ and set $f'_V(a):= g'(0) \in  \cal L (V,Y)$.

We say that $f$ is {\it strictly differentiable at $a$ along $V$} if $f'_V(a)$ exists and
\begin{equation}\label{stral}
\lim_{(x, \tilde x) \to (a,a),\ 0 \neq \tilde x -x \in V}\ \frac{|f(\tilde x) - f(x) - f'_V(a) (\tilde x -x)|}
{|\tilde x -x|} = 0.
\end{equation}
\end{definition}
\begin{remark}\label{along}
\begin{enumerate}
\item[(i)] In the above definition, some authors write ``with respect to $V$''
 or ``in the direction of $V$'' instead of ``along $V$''.
\item[(ii)] The standard strict differentiability coincides with strict differentiability
 along $V:= X$. Note that some authors by ``strict differentiability'' means (a weaker) ``G\^ ateaux
 strict differentiability'' which is stronger than G\^ ateaux differentiability.
\item[(iii)] Condition \eqref{stral} can be clearly rewritten as
\begin{equation}\label{stral2}
\forall \ep>0\, \exists \delta>0\, \forall v  \in V:\ \ \{x, x+v\} \subset B(a, \delta)  \Rightarrow
 |f(x+v)- f(x) - f'_V(a) v | \leq \ep |v|.
\end{equation}
\end{enumerate}
\end{remark}
The notions of partial Fr\' echet differentiability and strict partial differentiability are factually special cases of notions of ``directional'' and ``strict directional'' differentiability along a subspace.
 If  $X_1,\dots, X_n$ are Banach spaces, we consider the Banach space  $X:= X_1\times\cdots \times X_n$
 (equipped with the maximum norm). In the following definition, we consider, if $1\leq i\leq n$ is given, $X_i$ as a subspace of $X$
 (identifying, as usual, $x_i \in X_i$ with $(0,\dots,0,x_i,0,\dots,0)\in X$).

\begin{definition}\label{part}
 Let  $X_1,\dots,X_n$ and $Y$ be Banach spaces, $X:= X_1\times \cdots \times X_n$, $G \subset X$ an open set,
 $a \in G$ and $f:G \to Y$ a mapping. Then, for $i=1,\dots,n$,
\begin{enumerate}
\item[(i)]
we set  $f'_i(a):= f'_{X_i}(a)$ and call it (if it exists) {\it partial Fr\' echet derivative of $f$ at $a$
 with respect to the $i$ th variable}
 and
\item[(ii)]
 we say that $f$ is {\it partially strictly differentiable at $a$   with respect to the $i$ th variable}
 if $f$ is strictly differentiable at $a$ along $X_i$.
\end{enumerate}
\end{definition}

\section{Borel type of the set of points where a partial derivative exists}\label{borel}

If $X$, $Y$ are Banach spaces, $G\subset X$ an open set and $f: G \to Y$ an arbitrary mapping, then
 (see \cite[Theorem 2]{Za91})
\begin{equation}\label{df}
\text{the set $D(f)$  of all $x \in G$
 at which $f$ is Fr\' echet differentiable  is an $F_{\sigma \delta}$ set.}
\end{equation}  
This result was proved in  \cite{Za91} using a characterization of Fr\' echet differentiability points  (the proof in \cite[Corollary 3.5.5]{LPT} is quite different). 
Applying this characterization to partial functions, we  immediately obtain a characterization (Lemma \ref{zobchar} below) of points at which  $f$ has a partial Fr\' echet derivative. We need the following notation.
\begin{definition}\label{dvetc}
Let $X_1$, $X_2$, $Y$ be Banach spaces, $G \subset X:= X_1 \times X_2$ an open set and $f: G \to Y$ a mapping. 
\begin{enumerate}
\item[(i)]
We denote by $C_1(f)$ the set of all points $x= (x_1,x_2)\in G$ at which $f$ is continuous 
 with respect to the first coordinate (i.e. $f(\cdot,x_2)$ is continuous at $x_1$).
\item[(ii)]
For $c>0$, $\ep>0$ and $\delta>0$, denote by
 $D_1(f,c,\ep,\delta)$  the set of all points $x=(x_1,x_2) \in G$ such that
\begin{equation}\label{rme}
\left| \frac{f(z+kv,x_2) - f(z,x_2)}{k} - \frac{f(z,x_2)-f(z-hv,x_2)}{h}\right| \leq \ep
\end{equation} 
whenever  $v \in X_1$, $|v|=1$, $h>0$, $k>0$, $z\in X_1$,  $z\in B(x_1, \delta)$, 
$z-hv \in B(x_1, \delta)$, $z+kv \in B(x_1, \delta)$ and $\min(h,k) > c |z-x_1|$.
\end{enumerate}
\end{definition}
Then, by the definitions, $x=(x_1,x_2) \in D_1(f,c,\ep,\delta)$ if and only if $x_1 \in D(F,c,\ep,\delta)$,
  where  $F:= f(\cdot,x_2)$ and $ D(F,c,\ep,\delta)$ is as in \cite[Definition 3]{Za91}.
 Since  $f'_1(x)$ exists if and only $F'(x_1)$ exists, 
 \cite[Theorem 1 and Note 2]{Za91} immediately imply the following result.
\begin{lemma}\label{zobchar}
  Let $X_1$, $X_2$, $Y$, $G$ and $f$  be as in Definition \ref{dvetc} and $x\in G$. Then the 
 following conditions are equivalent.
\begin{enumerate}
\item [(i)]
 $f'_1(x)$ exists.
\item [(ii)]
 $x \in C_1(f) \cap  \bigcap_{c>0}\bigcap_{\ep>0}\bigcup_{\delta>0} D_1(f,c,\ep,\delta). $
\item [(iii)]
$x \in  C_1(f) \cap  \bigcap_{n \in \N}  \bigcup_{p \in \N}  \, D_1(f,1/n, 1/n, 1/p).$
\end{enumerate}
\end{lemma}  

Using Lemma \ref{zobchar}, we will prove the following proposition generalizing the corresponding result 
 on  real functions in $\R^2$
which was  proved in \cite{MP} by a quite different  elementary  method. This proposition will be used in the proofs
 of Proposition \ref{gdpf} and Theorem  \ref{sousepred}.

\begin{proposition}\label{borpart}
Let $X_1$, $X_2$, $Y$ be Banach spaces, $G\subset X:= X_1 \times X_2$ an open set and $f: G \to Y$ a mapping which is 
 continuous with respect to the second variable (i.e. all partial mappings $f(a_1,\cdot)$, $(a_1,a_2) \in G$,
 are continuous). Then the set
$ D_1(f):= \{x \in X:\ f_1'(x) \ \text{exists}\}$
 is an $F_{\sigma \delta\sigma}$ set. If $G=X$, then $D_1(f)$ is an $F_{\sigma \delta}$ set.
\end{proposition}
\begin{proof}
For each $k \in \N$, set  $F_k:= \{x \in G:\ \dist(x, X \setminus G) \geq 1/k\}$ if $X \neq G$ and 
 $F_k:= X$  if $X=G$. Clearly each $F_k$ is a closed set and $G= \bigcup_{k \in \N} F_k$. 
 So by Lemma \ref{zobchar} we obtain
\begin{equation}\label{char1}
D_1(f)= \bigcup_{k \in \N} (F_k \cap D_1(f)) =   \bigcup_{k \in \N} \left(F_k \cap  C_1(f) \cap  
\bigcap_{n \in \N}  \bigcup_{p \in \N}  \, D_1(f,1/n, 1/n, 1/p)\right).
\end{equation}
 Consequently it is sufficient to prove that, for each $k \in \N$, both
\begin{equation}\label{fkcj}
F_k \cap C_1(f)\ \ \text{is an}\ \ F_{\sigma \delta}\ \ \text{set}
\end{equation}
and
\begin{equation}\label{fkdj}
F_k \cap \bigcap_{n \in \N}  \bigcup_{p \in \N}  \, D_1(f,1/n, 1/n, 1/p)\ \ \text{is an}\ \ F_{\sigma \delta}\ \ \text{set}.
\end{equation}
First we will prove \eqref{fkcj}. To this end, fix an arbitrary $k \in \N$ and denote, 
 for $m,j \in \N$,
$$C_{m,j}:= \{x=(x_1,x_2) \in G:\ |f(t,x_2)- f(\tau,x_2)| \leq 1/m\ \text{whenever}\ 
 |t-x_1|< 1/j,\ |\tau-x_1|< 1/j\},$$
 and observe that
$$  C_1(f) = \bigcap_{m=1}^{\infty} \bigcup_{j=k}^{\infty} C_{m,j},\ \  F_k \cap C_1(f)= 
\bigcap_{m=1}^{\infty} \bigcup_{j=k}^{\infty}(F_k \cap C_{m,j}).$$
So it is sufficient to prove that  $F_k \cap C_{m,j}$ is a  closed set whenever  $j\geq k$. To this end, fix arbitrary $m,j \in \N$ with $j\geq k$
 and suppose that  $(x_1^i,x_2^i) \in F_k \cap C_{m,j}$, $i\in \N$, and  $ (x_1^i,x_2^i) \to (x_1,x_2),\ i\to \infty$. Then $(x_1,x_2) \in F_k$ since $F_k$ is closed.
  To prove  $(x_1,x_2) \in C_{m,j}$, consider arbitrary $t, \tau \in X_1$ with $|t-x_1| < 1/j$, $|\tau-x_1|< 1/j$.
	 Then, for all sufficiently large $i$, we have $|t-x_1^i| < 1/j$, $|\tau-x_1^i|< 1/j$, and consequently
	 $|f(t,x_2^i) - f(\tau, x_2^i)| \leq 1/m$. 
	Since $(x_1,x_2) \in F_k$ and $j\geq k$, we obtain that $(t,x_2) \in G$ and   $(\tau,x_2) \in G$.
	Since $f$ is continuous with respect to the second
	 variable, $f(t,x_2^i) \to f(t, x_2)$, $f(\tau,x_2^i) \to f(\tau, x_2)$, and therefore
	  $|f(t,x_2) - f(\tau, x_2)| \leq 1/m$. So we obtain  that $(x_1,x_2) \in  C_{m,j}$ and we are done.
		
		To prove \eqref{fkdj}, fix an arbitrary $k \in \N$. Since clearly 
		$$D_1(f,1/n, 1/n, 1/p_1) \subset D_1(f,1/n, 1/n, 1/p_2)\ \ \text{ whenever}\ \  p_1 \leq p_2,$$
		we have
		$$ F_k \cap \bigcap_{n \in \N}  \bigcup_{p \in \N}  \, D_1(f,1/n, 1/n, 1/p)= 
		\bigcap_{n=1}^{\infty}  \bigcup_{p=k}^{\infty} (F_k \cap D_1(f,1/n, 1/n, 1/p)).$$
		So it is sufficient to prove that  $F_k \cap D_1(f,1/n, 1/n, 1/p)$ is a  closed set whenever  $p\geq k$. To this end, fix arbitrary $n,p \in \N$ with $p\geq k$ 
 and suppose that  $x^i=(x_1^i,x_2^i) \in  F_k \cap D_1(f,1/n, 1/n, 1/p) $, $i \in \N$, and $ x^i=(x_1^i,x_2^i) \to x=(x_1,x_2),\ i\to \infty$. Then $(x_1,x_2) \in F_k$ since $F_k$ is closed.
To prove that $x=(x_1,x_2) \in  D_1(f,1/n, 1/n, 1/p) $,
 consider  arbitrary  $v\in X_1$ with $|v|=1$, reals  $h>0$, $k>0$ and  $z\in X_1$ such that
 $z \in B(x_1,1/p)$, $z-hv \in B(x_1,1/p)$, $z+kv \in B(x_1,1/p)$ and $\min(h,k) >  (1/n) |z-x_1|$.
 Since $(x_1,x_2) \in F_k$ and $p\geq k$, we obtain that $(z,x_2) \in G$, $(z+kv,x_2)\in G$  and
  $(z-hv,x_2)\in G$.
 Our aim is to prove that
\begin{equation}\label{chceme}
 \left| \frac{f(z+kv,x_2) - f(z,x_2)}{k} - \frac{f(z,x_2)-f(z-hv,x_2)}{h}\right| \leq  \frac{1}{n}.
\end{equation}
 Since $x_1^i \to x_1$, for all sufficiently large $i$, we have
 $z \in B(x_1^i,1/p)$, $z-hv \in B(x_1^i,1/p)$, $z+kv \in B(x_1^i,1/p)$ and $\min(h,k) > (1/n) |z-x_1^i|$, which
 together with   $x^i \in  D_1(f,1/n, 1/n, 1/p)$ implies
 \begin{equation}\label{vime}
 \left| \frac{f(z+kv,x_2^i) - f(z,x_2^i)}{k} - \frac{f(z,x_2^i)-f(z-hv,x_2^i)}{h}\right| \leq  \frac{1}{n}.
\end{equation}
  Since $f$ is continuous with respect to the second
	 variable, we have $f(z,x_2^i) \to f(z,x_2)$, $f(z+kv,x_2^i) \to f(z+kv,x_2)$,
	$f(z-hv,x_2^i) \to  f(z-hv,x_2)$, and consequently
	\eqref{vime} implies \eqref{chceme}. So we obtain that  $x \in  D_1(f,1/n, 1/n, 1/p) $ and
	  we are done.
\end{proof}

\section{Case of continuous mappings}\label{cont}

\subsection{Strict directional differentiability via Fr\' echet directional differentiability}

A well-known theorem (see e.g. \cite[Theorem B, p. 476]{Za91}) asserts that Fr\' echet differentiability at a point $x$ of an {\it arbitrary}
 mapping  $f:X \to Y$, where $X$, $Y$ are arbitrary Banach spaces, generically implies strict differentiability
  of $f$ at $x$. The following result which is a partial generalization of this theorem will be applied in the proof
	 of Theorem \ref{sousp}.
\begin{theorem}\label{dirfrstr}
Let $X$, $Y$ be Banach spaces and let $V$ be a  subspace of $X$ such that the space 
  $\cal L (V,Y)$ is  separable. Let $G\subset X$ be an open set and $f: G \to Y$  a continuous mapping. Then the set $A$
	 of all $a \in G$ such that $f'_V(a)$ exists and $f$ is not strictly differentiable at $a$ along
	 $V$ is a first category set.
	\end{theorem}
	\begin{proof}
	Choose a countable dense subset $\Phi$ of $\cal L(V, Y)$ and consider an arbitrary point $a\in A$.
	By  \eqref{stral2}  we can choose $n \in \N$ such that
	\begin{equation}\label{jhvn}
	\forall\, \delta>0\, \exists\,  x\in X\,  \exists v \in V:\ \ \{x, x+v\} \subset B(a,\delta),\ 
	 |f(x+v) -f(x) - f'_V(a) v| >  \frac{4}{n} |v|.
	\end{equation}
	Further choose $p\in \N$  such that
	\begin{equation}\label{thvn}
	|f(a+v)- f(a) - f'_V(a) v| \leq \frac{1}{n} |v|\ \ \ \text{whenever}\ \ \ v\in V,\ |v| \leq \frac{1}{p}
	\end{equation}
	 and choose $\vf \in \Phi$ such that
	\begin{equation}\label{dhvn}
	|f'_V(a)- \vf| \leq \frac{1}{n}.
	\end{equation}
	Denote, for each $n\in \N$, $p\in \N$ and $\vf \in \Phi$, by $A_{n,p,\vf}$ the set of all $a \in A$
	 for which the conditions \eqref{jhvn}, \eqref{thvn} and \eqref{dhvn} hold. Then 
	$A = \bigcup \{ A_{n,p,\vf}:\ n,p \in \N, \vf \in \Phi\}$ and so it is sufficient to show
	 that all sets $A_{n,p,\vf}$ are nowhere dense. 
	
	To this end, suppose to the contrary that, for some fixed $n$, $p$, $\vf$, the set $A_{n,p,\vf}$ is not nowhere dense.
	 Then there exist $z \in A_{n,p,\vf}$ and $\omega >0$ such that $A_{n,p,\vf}$ is dense in $B(z, \omega)$.
	  Now observe that, using \eqref{thvn} and \eqref{dhvn}, we easily obtain that
		\begin{equation}\label{trojn}
		|f(a+v) -f(a) - \vf(v)| \leq \frac{2}{n} |v|\ \ \text{whenever}\ \ \ a \in A_{n,p,\vf},\ v\in V,\ |v| \leq \frac{1}{p}.
		\end{equation}
		Applying \eqref{jhvn} to $a:= z$ with $\delta:= \min(\omega, 1/2p)$, we can choose
	$x \in B(z,\delta)$ and $v\in V$ such that $x+v \in  B(z,\delta)$ and
	\begin{equation}\label{dctn}
	|f(x+v) -f(x) - f'_V(z) v | > \frac{4}{n} |v|.
	\end{equation}
	Clearly $|v|< 1/p$. 
	By the choice of $z$, $\omega$ and $\delta$, there exist $a_k \in A_{n,p,\vf}$, $k \in \N$,
	 such that $a_k \to x$, and therefore $a_k +v \to x+v$. Since  $|v|< 1/p$, using \eqref{trojn} to $a:=a_k$, $k \in \N$,
	 and continuity of $f$, we easily obtain  $|f(x+v) -f(x) - \vf(v)| \leq \frac{2}{n} |v|$. This inequality
		 together with $|f'_V(z) - \vf|\leq 1/n$ implies 
		$|f(x+v) -f(x) - f'_V(z) v | \leq \frac{3}{n} |v|$ which contradicts \eqref{dctn}.
	\end{proof}
	
	\begin{remark}\label{ljesep}
	Theorem \ref{dirfrstr} can be applied e.g. in the following cases:
	\smallskip
	
	\begin{enumerate}
	\item[(i)] $V$ is finite-dimensional and $Y$ is separable.
	\item[(ii)]  $V$ is a separable Asplund space and $Y$ is finite-dimensional.
	\item[(iii)]    $V = C(K)$ for a countable compact set and $Y$ is separable with the Radon-Nikod\' ym property.
	\item[(iv)]  $V$ is a closed subspace of $c_0$ and $Y$ is separable with the Radon-Nikod\' ym property.
	\item[(v)]  $V= \ell_p$, $Y= \ell_q$, $1 \leq q<p< \infty$.  
	\end{enumerate}
	
	 Indeed, in all these cases, the space  $\Li (V,Y)$ is separable. The most natural cases (i) and (ii)
	 are almost obvious. For the cases (iii) and (iv) see \cite[pp. 114--115]{LPT}. In the well-known case (v)
	 Pitt's theorem (see, e.g., \cite[Proposition 6.25]{FHZ}) says that  $\Li(V,Y)$ coincides with the space
	$\Ki(V,Y)$ of compact operators which is separable (e.g. by well-known \cite[Fact 5.4, p. 20]{VZ}).
	
	For some other cases involving classical
	 Banach spaces see \cite[Example 5.5]{VZ}.
	\end{remark}
	
	\begin{remark}\label{notbel}
	 I do not believe that the assumption that $\Li (V,Y)$ is separable can be omitted in Theorem \ref{dirfrstr},
	 but I do not know any counterexample.
	\end{remark}

	\subsection{Fr\' echet differentiability of continuous functions via partial Fr\' echet differentiability} 
	 \ 
	
	The main result of the present section  (Theorem \ref{sousp} below) is an almost immediate consequence
	of Theorem \ref{dirfrstr} and the following known result    
 (see \cite[Proposition 2.57]{Pe}, where ``strict differentiability'' is called ``circa-differentiability''). 
\begin{proposition}\label{penot} (\cite{Pe})
Let  $X_1,\dots,X_n$ and $Y$ be Banach spaces, $X:= X_1\times \cdots \times X_n$, $G \subset X$ an open set,
 $a \in G$ and $f:G \to Y$ a mapping. Let  $f$  be  partially  Fr\' echet differentiable  at $a$
 with respect to the $n$ th variable and let $f$ be strictly differentiable at $a$ with respect to the $j$ th
 variable for each $1\leq j \leq n-1$. Then $f$ is Fr\' echet differentiable at $a$.
\end{proposition}

\begin{theorem}\label{sousp}
	Let  $X_1,\dots,X_n$ and $Y$ be Banach spaces, $X:= X_1\times \cdots \times X_n$, $G \subset X$ an open set,
  and let $f:G \to Y$ be a continuous mapping. Suppose that  the spaces
 $\cal L (X_1, Y),\dots, \cal L (X_{n-1}, Y)  $ are separable. Then there exists a first
 category set $A \subset G$ such that, for all $x \in G \setminus A$, the following implication
 holds:
$$\text{$f$ has all Fr\' echet partial derivatives at $x$\ \ $\Rightarrow$\ \ $f$ is  Fr\' echet differentiable at 
$x$.} $$
\end{theorem}
\begin{proof}
Let $A_i$, $1\leq i \leq n-1$, be the set of all $x \in G$ such that $f'_{X_i}(x)$ exists and
	 $f$ is not strictly differentiable at $x$ along $X_i$ (recall that we identify $X_i$ with a subspace of $X$ by the usual way). By Theorem  \ref{dirfrstr}  each $A_i$ is a first category set
	 and consequently $A:= A_1\cup \dots \cup A_{n-1}$ is also a first category set. If  $x \in G \setminus A$
	 anf $f$ has all Fr\' echet partial derivatives at $x$, then $f$ is strictly differentiable at $x$ 
	 with respect to the $j$ th variable for each  $1\leq j \leq n-1$ and so $f$ is Fr\' echet  differentiable
	 at $x$ by Proposition \ref{penot}.
	\end{proof}
	
	\begin{remark}\label{pitt}
	 Probably the most interesting is the case of a real function $f$ (i.e. $Y=\R$). Then we assume
 that the dual space  $(X_i)^*$ is separable (i.e. $X_i$ is a separable Asplund space)	for each $1\leq i \leq n-1$. Using the method of separable reduction, we will show that the result holds also
 if   $X_i$, $1\leq i \leq n-1$, are general Asplund spaces  (see Corollary \ref{aspl} below).
A number of other concrete applications of Theorem \ref{sousp} can be easily obtained using the facts
 from Remark \ref{ljesep}.
	\end{remark}

	As an easy consequence of Theorem \ref{sousp} (and Proposition \ref{borpart}) we obtain the following
	 result on generic Fr\' echet differentiability of functions whose all partial functions are generically
	   Fr\' echet differentiable.
		
		\begin{proposition}\label{gdpf}
		Let  $X_1,\dots,X_n$ and $Y$ be Banach spaces, $X:= X_1\times \cdots \times X_n$, 
  and let $f:X \to Y$ be  continuous. Suppose that all $X_i$, $1\leq i \leq n$, are separable and the spaces
 $\cal L (X_1, Y),\dots, \cal L (X_{n-1}, Y)  $ are separable.
		Let each partial function
		$$    f(x_1,\dots, x_{i-1},\cdot, x_{i+1},\dots,x_n),\ \ 1 \leq i \leq n,$$
		be generically Fr\' echet differentiable on $X_i$. Then $f$ is generically Fr\' echet differentiable. 
		\end{proposition}
	\begin{proof}
	For each  $1 \leq i \leq n$,  denote
	$$   P_i:= \{ x \in X:\ f'_i(x)\ \ \text{does not exist}\}.$$
	 Identifying by the natural way  $X$ with $Y_i \times X_i$, where 
	$Y_i:= X_1 \times \cdots \times X_{i-1} \times X_{i+1} \times \cdots \times X_n$,
	 observe that, by the assumptions, the set $\{t \in X_i: (y,t) \in P_i\}$
	 is a first category set in $X_i$ for each $y \in Y_i$. Since  $P_i$ is a Borel set
	 by Proposition \ref{borpart}, it has the Baire property and consequently the Kuratowski-Ulam theorem (see, e.g.,
	 \cite[Theorem 8.41]{Ke}) implies that 
	 $P_i$ is a first category set. Thus $f$ has all  Fr\' echet  partial derivatives outside the first category set
	 $P_1\cup\cdots\cup P_n$ and consequently $f$ is generically Fr\' echet differentiable
	 by Theorem \ref{sousp}. 
	 \end{proof}

	Similarly as Theorem \ref{sousp}, we obtain the following result.

\begin{proposition}\label{dirsp} 
Let  $X$, $Y$ be Banach spaces, $G \subset X$  an open set
  and $f:G \to Y$ a continuous mapping. Let $V_1$ be a subspace of $X$ such that the space
	  $\cal  L (V_1,Y)$ is separable. Then there exists a first category set $A \subset G$
	 such that, for each $x \in G \setminus A$, the following assertion holds:
	\smallskip
	
$(*)$\ \ \  If $f$ is Fr\' echet differentiable  at $x$ along $V_1$ and along some topological complement
	 $V_2^x$ of $V_1$, then $f$ is Fr\' echet differentiable  at $x$.
\end{proposition}
\begin{proof} 
 Let $A$ be the set of all $x \in G$ such that $f'_{V_1}(x)$ exists and
	 $f$ is not strictly differentiable at $x$ along $V_1$. By Theorem  \ref{dirfrstr},  $A$ is a first category set. Now fix an arbitrary $x \in G \setminus A$ and suppose that $f$ is Fr\' echet differentiable  at $x$ along $V_1$ and along some topological complement
	 $V_2^x$ of $V_1$. Further identify $X$ with  $\tilde X: = V_1 \times V_2^x$ 
	 by the canonical isomorphism. Then $f$ (considered on $\tilde X$)
	 is strictly differentiable at $x$ along  $V_1$ (considered as a subspace of $\tilde X$) and it is
	 Fr\' echet differentiable along $V_2^x$. So Proposition \ref{penot} implies that $f$ is Fr\' echet differentiable
	 at $x$.
	\end{proof}

\section{Lipschitz case}\label{lip}

 \begin{lemma}\label{obec}
Let $X$, $Y$ be Banach spaces, $G \subset X$ an open set and $f:G \to Y$ an arbitrary mapping. Let $V_1$ 
 be a  subspace of $X$ such that $\cal L (V_1,Y)$ is  separable.
  Then there exists a $\sigma$-porous set $A \subset G$ such that
 the following assertion holds for each $x \in G \setminus A$:

$(*)$\ \ 
Let $f$ be Lipschitz at $x$ and  Fr\' echet differentiable at $x$ along  $V_1$, and  let $V^x$ and  $V_2^x$
be subspaces
 of $X$ such that
 $V^x = V_1 \oplus V_2^x$ and $f$ is Fr\' echet differentiable at $x$ along $V_2^x$. Then $f$ is Fr\' echet differentiable at $x$ along $V^x$.
\end{lemma}
\begin{proof}
First choose a countable dense subset $\Phi$ of $\cal L(V_1, Y)$.
Further denote by $A$ the  set of all $x \in G$, for which assertion $(*)$ does not hold. 

 Then, for each $x \in A$, $f$ is Lipschitz at $x$, is Fr\' echet differentiable at $x$ along  $V_1$, 
  and we can fix spaces $V^x$ and  $V_2^x$  such that
 $V^x = V_1 \oplus V_2^x$, $f$ is Fr\' echet differentiable at $x$ along $V_2^x$ and $f$ is not Fr\' echet differentiable at $x$ along $V^x$.

Now consider an arbitrary fixed $x \in A$.
 We can write any $v \in V^x$  in a unique way as  $v= v_1^x + v_2^x$ with $v_1^x \in V_1$ and $v_2^x \in V_2^x$
 and choose $p\in \N$ such that
\begin{equation}\label{nopr}
  \max(|v_1^x|, |v_2^x|)  \leq p |v|\ \ \ \text{for each}\ \ \ v \in V^x.
	\end{equation}
	Further, since $f$ is Lipschitz at $x$,  we can choose  $l \in \N$ so big that
\begin{equation}\label{elip}
 |f(y)-f(x)| \leq  l |y-x|\ \ \text{whenever}\ \ |y-x| \leq 1/l.
\end{equation}
Define  $\Psi^x: V^x \to Y$ setting  $\Psi^x(v) = f'_{V_1}(x)v_1^x + f'_{V_2^x}(x)v_2^x,\ \ v \in V^x$.
Then $ \Psi^x \in \cal L (V^x,Y)$. Since $f'_{V^x}(x)$ does not exist, we can choose $n\in \N$ such that
\begin{equation}\label{limsup}
 \limsup_{v \to 0, v \in V^x}  \ \frac{|f(x+v) - f(x) - (f'_{V_1}(x)v_1^x + f'_{V_2^x}(x)v_2^x)|} {|v|} > \frac{7}{n}.
\end{equation}
Further choose $m \in \N$ such that
\begin{equation}\label{dervj}
|f(x+h_1)-f(x) - f'_{V_1}(x) h_1| \leq  \frac{1}{pn} |h_1|\ \ \text{whenever}\ \ h_1\in V_1, |h_1| < \frac{1}{m}
\end{equation}
and
\begin{equation}\label{dervd}
|f(x+h_2)-f(x) - f'_{V_2^x}(x) h_2| \leq  \frac{1}{pn} |h_2|\ \ \text{whenever}\ \ h_2\in V_2^x, |h_2| < \frac{1}{m}.
\end{equation}
Finally choose  $\vf \in \Phi$ such that
\begin{equation}\label{vf}
 |f'_{V_1}(x)- \vf| \leq \frac{1}{pn}.
\end{equation}
For $p,l,n,m \in \N$ and $\vf\in \Phi$, denote by $A_{p,l,n,m,\vf}$ the set of all $x\in A$ for which
conditions  \eqref{nopr}, \eqref{elip}, \eqref{limsup}, \eqref{dervj}, \eqref{dervd}, \eqref{vf} hold.
Then  $A = \bigcup \{ A_{p,l,n,m,\vf}:\ p,l,n,m \in \N,\ \vf \in \Phi\}$ and thus it is sufficient
to prove that, for each fixed  $p,l,n,m \in \N$ and $\vf\in \Phi$, the set  $A^*:= A_{p,l,n,m,\vf}$
is porous. 

Suppose to the contrary that  $x \in A^*$ such that $A^*$ is not porous at $x$ is given. Then, by Definition \ref{dpor}, we can choose  
 $0< \delta < 1$ such that
\begin{equation}\label{por}
B\left(t,\frac{|t-x|}{2pln}\right) \cap A^* \neq \emptyset\ \ \text{whenever}\ \ 
t \in B(x,\delta)\setminus \{x\}.
\end{equation}
By \eqref{limsup} we can choose $ v\in V^x$ such that 
\begin{equation}\label{vmal}
0<|v| < \frac{\delta}{pm}<1
\end{equation}
and
\begin{equation}\label{zlimsup}
D:= |f(x+v) - f(x) - (f'_{V_1}(x)v_1^x + f'_{V_2^x}(x)v_2^x)|  > \frac{7}{n} |v|.
\end{equation}
 By \eqref{nopr} and \eqref{vmal} we obtain 
\begin{equation}\label{nove}
 \max(|v_1^x|, |v_2^x|)\leq p |v| < p\cdot \frac{\delta}{pm} = \frac{\delta}{m} < \frac{1}{m}
\end{equation}
 and so \eqref{dervd} and \eqref{nopr} imply
\begin{equation}\label{drdif}
|f(x+v_2^x)-f(x) - f'_{V_2}(x)v_2^x | \leq  \frac{1}{pn} |v_2^x | \leq  \frac{1}{n} |v |.
\end{equation}
By  \eqref{nove} we have $|v_2^x| < \delta$. Further we have $v_2^x\neq 0$, since otherwise $v= v_1^x$
 and so \eqref{dervj} with \eqref{nove} imply $D \leq (1/n) |v|$ which contradicts  \eqref{zlimsup}.
Thus we can apply \eqref{por} to $t:= x + v_2^x$ and obtain a point $y \in A^*$ such that
\begin{equation}\label{vzdy}
| (x + v_2^x)-y|  < \frac{|v_2^x|}{2pln} \leq \frac{p |v|}{2pln} = \frac{|v|}{2ln} \leq \frac{1}{l}
\end{equation}
 (we have used also \eqref{nopr} and \eqref{vmal}). Similarly, since $0< |v| < \delta$ by \eqref{vmal},
 we obtain  by \eqref{por}  a point $z \in A^*$ such that
\begin{equation}\label{vzdz}
| (x + v)-z| < \frac{|v|}{2pln}  \leq \frac{1}{l}.
\end{equation}
Since $y \in A^*$ and $z \in A^*$, by \eqref{vzdy}, \eqref{vzdz} and \eqref{elip} we obtain
\begin{equation}\label{dify}
|f(x + v_2^x)- f(y)| \leq l\cdot | (x + v_2^x)-y|\leq l\cdot \frac{ |v|}{2ln} \leq \frac{1}{n} |v |
\end{equation}
and
\begin{equation}\label{difz}
|f(x + v)- f(z)| \leq l\cdot | (x + v)-z|\leq l\cdot \frac{|v|}{2pln} \leq \frac{1}{n} |v |.
\end{equation}
Since $y \in A^*$  and  $|v_1^x| < 1/m$  by \eqref{nove}, using \eqref{dervj}
 and \eqref{nopr} we obtain
\begin{equation}\label{dervy}
 |f(y + v_1^x) - f(y) - f'_{V_1}(y) v_1^x| \leq \frac{1}{pn} |v_1^x|  \leq \frac{1}{n} |v |.
\end{equation}
We have  $|(x+v) -( y + v_1^x)|  =   | (x + v_2^x)-y| \leq (2ln)^{-1} |v| $ by \eqref{vzdy} and consequently
 \begin{equation}\label{zposy}
     |z- ( y + v_1^x)| \leq  | (x + v)-z|  +    |(x+v) -( y + v_1^x)| \leq \frac{|v|}{2pln} + \frac{|v|}{2ln} 
 \leq  \frac{|v|}{ln}  \leq \frac{1}{l}
\end{equation}
 by \eqref{vzdz}.
Since $z \in A^*$, we obtain by \eqref{zposy} and \eqref{elip}
$$
|f(z) - f ( y + v_1^x)| \leq l |z- ( y + v_1^x)|\leq l\cdot  \frac{|v|}{ln} = \frac{1}{n} |v |,
$$
which together with \eqref{difz} gives
\begin{equation}\label{fvposy}
|f(x+v) - f ( y + v_1^x)| \leq |f(x + v)- f(z)|  + |f(z) - f ( y + v_1^x)| \leq \frac{2}{n} |v |.
\end{equation}
Since $y \in A^*$, we obtain by \eqref{vf} (with $x:=y$) $|f'_{V_1}(y) - \vf| \leq    (pn)^{-1}$, which
 together with \eqref{vf} gives
$|f'_{V_1}(y) - f'_{V_1}(x)| \leq  2 (pn)^{-1}$. Using also \eqref{nopr}, we obtain
\begin{equation}\label{rdervj}
| f'_{V_1}(y)v_1^x - f'_{V_1}(x)v_1^x|  \leq \frac{2}{pn}\cdot |v_1^x| \leq \frac{2}{n} |v |.
\end{equation}
Using  \eqref{drdif}, we obtain the following upper estimate of $D$ (from \eqref{zlimsup}).
\begin{multline}\label{odhsh}
D:= |f(x+v) - f(x) - (f'_{V_1}(x)v_1^x + f'_{V_2^x}(x)v_2^x)|\\ = |(f(x+v) - f(x+v_2^x) - f'_{V_1}(x)v_1^x )
 + (f(x+v_2^x) -f(x) - f'_{V_2^x}(x)v_2^x)| \\
\leq |f(x+v) - f(x+v_2^x) - f'_{V_1}(x)v_1^x | + \frac{1}{n} |v |.
\end{multline}
Further, using  \eqref{dervy}, \eqref{fvposy}, \eqref{dify} and \eqref{rdervj}, we obtain
\begin{multline*}
|f(x+v) - f(x+v_2^x) - f'_{V_1}(x)v_1^x |\\ \leq |f(y+v_1^x) - f(y) - f'_{V_1}(y) v_1^x| +
 |f(x+v) - f ( y + v_1^x)| +  |f(y) -f(x + v_2^x)| + | f'_{V_1}(y)v_1^x - f'_{V_1}(x)v_1^x| \\
 \leq   \frac{1}{n} |v | + \frac{2}{n} |v | + \frac{1}{n} |v | +  \frac{2}{n} |v | = \frac{6}{n} |v |.
\end{multline*}
Thus \eqref{odhsh} gives $D\leq \frac{7}{n} |v |$ which contradicts \eqref{zlimsup}.
\end{proof}   
By induction, we easily infer from Lemma \ref{obec} the following analogon of Theorem \ref{sousp}.

\begin{theorem}\label{soulip}
	Let  $X_1,\dots,X_n$ and $Y$ be Banach spaces, $X:= X_1\times \cdots \times X_n$, 
	$G \subset X$ an open set
  and $f:G \to Y$ an  arbitrary mapping. Let the spaces
 $\cal L (X_1, Y),\dots, \cal L (X_{n-1}, Y)  $ be separable.    Then there exists a $\sigma$-porous
 set $A \subset G$ such that, for all $x \in G \setminus A$, the following  implication holds. 
\begin{multline*}
\text{$(*)$\ \ \ $f$ is Lipschitz at $x$ and has all Fr\' echet partial derivatives at $x$}\\
 \Rightarrow\ \ \text{$f$ is  Fr\' echet differentiable at 
$x$.} 
\end{multline*}
\end{theorem}
\begin{proof}
We will proceed by induction on $n$. For $n=2$ the assertion immediately follows from Lemma \ref{obec}
	 used for $V_1:= X_1$ and $V_2^x:= X_2$.
	
	Now suppose that $n\geq3$ and ``the theorem holds for $n:=n-1$''. 
	Further suppose that $X_1,\dots,X_n$, $Y$, $G$ and $f$ which satisfy the assumptions of the theorem are  
		 given.  Since $\cal L (X_{n-1}, Y) $  is separable, we can use Lemma \ref{obec} with $V_1:= X_{n-1}$,
		 and obtain a  $\sigma$-porous set $A_1 \subset G$ such that, for each  $x \in G \setminus A_1$, the following assertion holds:
 \smallskip
	
	($\alpha_1$)\ \ \ If $f$ is Lipschitz at $x$ and 	 Fr\' echet differentiable
	 at $x$ along $X_{n-1}$ and along $X_n$, then $f$ is 	 Fr\' echet differentiable
	 at $x$ along the space $\tilde X_{n-1} : = X_{n-1} \times X_n$.
	\smallskip
	
	Now we identify $X$ with $\tilde X:= X_1 \times \cdots \times X_{n-2} \times \tilde X_{n-1}$ by the usual
	 way. By the inductive assumption, there exists  a  $\sigma$-porous set $A_2 \subset G$ such that, for each  $x \in G \setminus A_2$, the following assertion holds:
 \smallskip
	
	($\alpha_2$)\ \ \ If $f$ is Lipschitz at $x$ and 	 Fr\' echet differentiable at $x$ along all
	 spaces   $X_1,\dots, X_{n-2}, \tilde X_{n-1}$, then $f$ is  Fr\' echet differentiable at $x$
	(in $\tilde X = X$). 
	\smallskip
	
	Setting $A:= A_1 \cup A_2$, and using for each $x \in G\setminus A$ the validity of ($\alpha_1$)
	and ($\alpha_2$), we obtain   
	 that implication $(*)$ holds for each $x \in G\setminus A$.
	\end{proof}
		
		As an immedite consequence of Lemma \ref{obec}, we obtain the following analogon of Proposition \ref{dirsp}.
\begin{proposition}\label{dirlip} 
Let  $X$, $Y$ be Banach spaces, $G \subset X$  an open set
  and $f:G \to Y$ an arbitrary mapping. Let $V_1$ be a subspace of $X$ such that the space
	  $\cal  L (V_1,Y)$ is separable. Then there exists a $\sigma$-porous set $A \subset G$
	 such that, for each $x \in G \setminus A$, the following assertion holds:
	\smallskip
	
	$(*)$\ \ \  If $f$ is Lipschitz at $x$, is Fr\' echet differentiable  at $x$ along $V_1$ and along some topological complement
	 $V_2^x$ of $V_1$, then $f$ is Fr\' echet differentiable  at $x$.
\end{proposition}

In connection with Propositions \ref{dirsp} and \ref{dirlip}, it is natural to ask, for which Banach spaces
 $X$, $Y$ the following statement holds.
\smallskip

$(S)$  \ \ Let $f: X \to Y$ be continuous (resp. Lipschitz). Denote by $E_f$ the set of all points $x\in X$
 at which there exist a subspace $V^x$ of $X$ and its topological complement $W^x$ such that $f$ 
  is Fr\' echet differentiable  at $x$ both along $V^x$ and $W^x$ but is not Fr\' echet differentiable  at $x$.
	 Then $E_f$ is a first category set (resp. a $\sigma$-porous set).
	\smallskip
	
	If $ \dim X < \infty$ and $Y=\R$, then the ``continuous part'' of $(S)$ holds; it easily follows from
	\cite{I}. Further, if $ \dim X < \infty$, then the ``Lipschitz part'' of $(S)$
	 easily follows from  \cite[Theorem 2]{PZ}
	 (and also from Corollary \ref{fregat} below). 
	
	I conjecture that
	 no part of  $(S)$ holds if $X$ is an infinite-dimensional space, but I do not know any counterexample. 

 For the weaker version of  $(S)$ (which we obtain demanding that $V^x$ in the definition of $E_f$ is finite-dimensional), see Corollary   \ref{fregat} below. 
It is an immediate consequence of  Proposition \ref{kkodsx} below, which
  can be of an independent interest. In its proof we use substantially Lemma \ref{obec} and the following result which
	 immediately follows from \cite[Corollary 3.4]{Za14} (since each $\sigma$-directionally porous set is clearly $\sigma$-porous).
	\begin{proposition}[{\cite{Za14}}]\label{liga}\ \ \   
     Let $X$ be a separable Banach space, $Y$ a Banach space, $G \subset X$ an open set, and $f: G \to Y$ an arbitrary mapping. 
		Then there exists a
 $\sigma$-porous set $A \subset G$ such that, for each $x \in G \setminus A$,
		 the following assertion holds.
      \smallskip
			
  $(*)$\ \  \ If $f$ is Lipschitz at $x$ and the one-sided directional derivative $f'_{+}(x,u)$ exists in all directions $u$ from a set $S_x \subset X$
  whose linear span is dense in $X$, then $f$ is G\^ ateaux differentiable at $x$.
          \end{proposition}

\begin{proposition}\label{kkodsx}
Let $X$, $Y$ be separable Banach spaces, $G \subset X$ an open set and $f: G \to Y$ an arbitrary mapping. Then there exists a
 $\sigma$-porous set $A \subset G$ such that, for each $x \in G \setminus A$, the following assertion holds.
\smallskip

$(**)$\ \ \ Let $f$ be Lipschitz at $x$ and Fr\'echet differentiable at $x$ along a subspace $M^x$ of finite codimension, and
 let there exist a set $S_x\subset X$ such that $\lin S_x$ is dense in $X$ and $f'_+(x,u)$ exists for all $u \in S_x$.
 Then $f$ is Fr\' echet differentiable at $x$.
\end{proposition}
\begin{proof}
 For the given $f$, we define $A$
as the set of all  $x \in G$ for which assertion $(**)$ does not hold; we will prove that $A$ is $\sigma$-porous.
 
Choose a dense countable set $D \subset X$ and, for each point $x \in A$, set
\begin{multline*}\text{ $k_x:= \inf\{ k \in \N \cup \{0\}: \ f$\  is  Fr\' echet differentiable at \ $x$}\\ \ \text{along a subspace $M^x$ of codimension $k\}$ },
\end{multline*}
(where we adopt the convention that $X$ has codimension $0$). Note that, if $x \in A$, then clearly
   $1 \leq k_x < \infty$.
Further, for each $x \in A$,  choose a subspace $M^x$ of $X$ of codimension $k_x$ such that $f$ is  Fr\' echet differentiable at $x$ along
  $M^x$ and then choose a vector $0 \neq v_x \in D \setminus M^x$. 
	 For each $ k \in \N $  and $v \in D$, set
	$$ A_{k,v}:= \{ x \in A: \ k_x=k,\ v_x = v\}.$$ 
	 Then  $A= \bigcup \{  A_{k,v}:\ k \in \N,\ v \in D\}$ and thus it is sufficient to prove that, for each fixed $k \in \N$
	 and $ v \in D$, the set $ A_{k,v}$ is $\sigma$-porous.
	
	To this end, set $V_1:= \lin \{v\}$ and observe that $\cal L (V_1,Y)$ is separable. Let $\tilde A \subset G$
	 be a $\sigma$-porous set which corresponds to $V_1$ and $f$ by Lemma \ref{obec}; i.e.  
	\begin{equation}\label{atilda}
	\text{assertion $(*)$ from Lemma \ref{obec} holds for each $x \in G \setminus \tilde A$.}
	\end{equation}
	Further, let
	$\tilde{\tilde A}$
	be a $\sigma$-porous set which corresponds to $f$ by Proposition \ref{liga}. 
	
	Now it is sufficient to prove that $ A_{k,v} \subset \tilde A  \cup \tilde{\tilde A}$. To prove this inclusion,
	 suppose to the contrary that there exists a point $x \in  A_{k,v} \setminus (\tilde A  \cup \tilde{\tilde A})$.
	Set $V_2^x:= M^x$ and $V^x:= V_1 + V_2^x$. We know that $V_2^x$ has codimension $1 \leq k_x < \infty$ and
	 $V_1 \cap V_2^x = \{0\}$. Consequently $V^x$ is closed, $V^x= V_1 \oplus V_2^x$ 
	(see, e.g., \cite[Exercise 5.27 and Proposition 5.3]{FHZ})
	 and it is easy to see that $V^x$ has codimension  $k_x-1$.
 Since $x \in A$, we have that the assumptions of assertion $(**)$ hold and so also the assumptions of
 assertion $(*)$ of Proposition \ref{liga} hold. 
So, since $x \in G \setminus \tilde{\tilde A}$, by the choice of  $\tilde{\tilde A}$ we obtain that
 $f$ is G\^ ateaux differentiable at $x$, and consequently $f$ is Fr\' echet differentiable
 along $V_1$ at $x$. Since $f$ is Fr\' echet differentiable at $x$ along the space 
 $V_2^x=  M^x$ and $x \in G \setminus \tilde A$, we can use 
  \eqref{atilda} and obtain that $f$ is   Fr\' echet differentiable
 along the space $V^x$ of codimension $k_x - 1$  which contradicts the definition of $k_x$.
\end{proof}

As  immediate consequences, we obtain the following results.

\begin{corollary}\label{fregat}
Let $X$, $Y$ be separable Banach spaces, $G \subset X$ an open set and $f: G \to Y$ an arbitrary mapping. Then there exists a
 $\sigma$-porous set $A \subset G$ such that the following assertion holds.
\smallskip

$(*)$\ \ \ Let $f$ be Lipschitz at $x$ and let there exist a finite-dimensional subspace  $V^x$ of $X$ and
 its topological complement $W^x$ of $X$
 such that $f$ is  Fr\' echet differentiable at $x$  along $V^x$
 and  $W^x$. 
 Then $f$ is Fr\' echet differentiable at $x$.
\end{corollary}

\begin{corollary}\label{gat}
Let $X$, $Y$ be separable Banach spaces,  $G \subset X$ an open set and $f: G \to Y$  a Lipschitz mapping.  Then the set $A$
 of all points $x \in G$ at which  $f$ is  G\^ ateaux differentiable, is  Fr\' echet differentiable
 along a closed subspace of a finite codimension but is not Fr\' echet
 differentiable, is $\sigma$-porous.
\end{corollary}
\begin{remark}\label{gatrem}
The set $A$ from Corollary \ref{gat} can be nonempty.
To show this, set $X:= \ell_2$, $Y:= \R$ and denote by $e_n,\ n\in \N,$ the canonical basis vectors in $\ell_2$.
 Set  
$$p_n:= \frac{1}{n} (e_1+ e_n),\ n\geq 2, \ F:= X \setminus \bigcup_{n=2}^{\infty}  B(p_n, \frac{1}{2n})
\ \ \text{and}\ \ f(x):= \dist (x,F),\ x \in X.$$
Then $f$ is Lipschitz on $G:=X$ and it is easy to check that $0 \in A$, where $A$ is as in Corollary \ref{gat}.
\end{remark}

\section{Results proved by the separable reduction method}\label{sepred}

In this section, we will use the well-known method of separable reduction. Namely, we will first prove ``the separable case''
 and from it we will obtain the ``nonseparable case'' using some
 known results  which say that some notions are ``separably determined in the sense of rich families''.
For the following notion of a ``rich family'' see e.g. \cite[p. 37]{LPT} or \cite{Cu2}.
\begin{definition}\label{rich}
Let $X$ be a Banach space. A family $\cal F$ of  separable subspaces of $X$ is called a {\it rich family}
 if:
 \begin{enumerate}
 \item[(R1)] If $V_i \in \cal F$ ($i \in \N$) and $V_1 \subset V_2\subset\dots$, then
  $\overline{\bigcup\{V_n: n \in \N\}} \in \cal F.$
 \item[(R2)] For each  separable subspace $V_0$ of $X$ there exists $V \in \cal F$ such that $V_0 \subset V$.
 \end{enumerate}
\end{definition}

A basic (easy) fact (see  e.g. \cite[Proposition 3.6.2]{LPT}) concerning rich families is the following.
\begin{lemma}\label{richpr}
Let $X$ be a Banach space and let $\{\cal F_n: n \in \N\}$ be rich families of  separable subspaces of $X$.
 Then $\cal F:= \bigcap\{\cal F_n: n \in \N\}$ is also a rich family of  separable subspaces of $X$.
\end{lemma}

We will use also the following simple fact which is a reformulation of \cite[Lemma 4.4]{Za12}.
\begin{lemma}\label{souc}
Let $X_1,\dots,X_n$ be  Banach spaces and  $X: = X_1 \times \cdots\times X_n$. Let $\F_k$ be a rich family of  separable subspaces of $X_k$, $1\leq k \leq n$. Then
$$ \F:= \{V_1\times\cdots\times V_n:\ V_k \in \F_k,\ 1\leq k \leq n\} $$
is a rich family of separable subspaces of $X$.
\end{lemma}

Much more difficult is the following result which says that Fr\' echet differentiability at a point is ``separably determined in the sense of rich families''.

\begin{theorem}[{\cite[Theorem 3.6.10]{LPT}}]\label{richfr}
Let $X$, $Y$ be Banach spaces, $G \subset X$ an open set and $f\colon G\to Y$ a mapping. Then there exists a rich family $\cal F$ of  separable subspaces of $X$ such that for every $V \in \cal F$, $f$ is Fr\' echet differentiable (with respect to $X$) at every $x \in V\cap G $, at which its restriction to $V\cap G$ is Fr\' echet differentiable (with respect to $V$).
\end{theorem}
(In fact, {\cite[Theorem 3.6.10]{LPT}} is formulated for $G=X$ only, but if we apply this formally weaker
 theorem to any extension $\tilde f$ of $f$ to $X$, we obtain the assertion of Theorem \ref{richfr}.)
 
The following result on ``separable determination of first category sets and $\sigma$-upper porous sets''
 were first proved in  \cite{Cu1} and \cite{CR} ``in the sense of suitable models'' and then transferred to the following result
 in \cite{Cu2}.  

\begin{theorem}[{\cite[Corollary 5]{Cu2}}]\label{richcuth} 
Let $X$ be a Banach space and $A \subset X$  a Souslin set. Then there exists a rich family $\cal F$ of  separable subspaces of $X$ such that for every $V \in \cal F$ we have
\begin{align*}
(i)\ \ \ \text{$A$ is of the first category in $X$} 
&\ \ \Longleftrightarrow\ \ 
\text{$A\cap V$ is of the first category in $V$},\\
(ii)\ \ \ \ \ \ \ \ \ \ \ \ \ \ \ \text{$A$ is $\sigma$-porous in $X$} 
&\ \ \Longleftrightarrow\ \ 
\text{$A\cap V$ is $\sigma$-porous
 in $V$}. 
\end{align*} 
\end{theorem}  

Recall that every Borel set in $X$ is Souslin. We will use also the following known fact.

\begin{theorem}[{\cite[Theorem 4.7]{Za12}}]\label{red}
Let $X$, $Y$ be Banach spaces,  $G \subset X$ an open set, and let $f:G \to Y$ be an arbitrary mapping. Then the following conditions are equivalent.
\begin{enumerate}
\item
$f$ is generically Fr\' echet differentiable.
\item
There exists a rich family $\F$ of  separable subspaces of $X$ such that $f|_{V\cap G}$ is generically 
 Fr\' echet differentiable (with respect to $V$) on $V \cap G$ for each $V \in \F$.
\end{enumerate}
\end{theorem}

Now we will prove, using Theorem \ref{sousp}, Theorem \ref{soulip} and the method of separable reduction,
 the following result which generalizes Theorem \ref{sousp} and partly generalizes Theorem \ref{soulip}.

	\begin{theorem}\label{sousepred}
	Let  $X_1,\dots,X_{n}$ be Banach spaces, $X:= X_1 \times \cdots \times X_n$ and  $G \subset X$
	 an open set.  Suppose that each space $\Li (\tilde X_i, \tilde Y)$
	  is separable  whenever $\tilde X_i$ is a separable subspace of $X_i$, $i=1,\dots, n-1$, and $\tilde Y$ is a separable subspace of $Y$.
		Let $f:G \to Y$ be a continuous (resp. Lipschitz) mapping.
		
		Then there exists a first category (resp. $\sigma$-porous) set $A \subset G$ such that, for each $x \in G \setminus A$, the following implication holds.
	$$\text{$f$ has all Fr\' echet partial derivatives at $x$\ \ $\Rightarrow$\ \ $f$ is  Fr\' echet differentiable at 
$x$.} $$
\end{theorem}
	\begin{proof}
	We will prove the ``continuous part'' and the ``Lipschitz part'' of the theorem together.

In the first step of the proof, we will prove the theorem in the special case when all spaces
 $X_1, \dots,  X_n$ are separable. In this case observe that the space $\tilde Y:= \overline { \lin f(G)} $
 is separable and therefore the spaces  $\Li (X_1, \tilde Y),\dots, \Li (X_{n-1}, \tilde Y)$
 are separable by the assumptions of the theorem. Further observe  that $f: G \to Y$ is partially Fr\' echet differentiable (resp. Fr\' echet differentiable) at $x\in G$ if and only if $f:G \to \tilde Y$ has this property.
 So, if  $f$ is continuous (resp. Lipschitz), then the existence of a first category (resp. $\sigma$-porous)
 set $A$ from the conclusion of the theorem follows from Theorem \ref{sousp} (resp. Theorem \ref{soulip}).

In the second step,  we will prove the general case using the method of separable reduction.
Denote by $A$ the set of all $x\in G$ at which all partial Fr\' echet derivatives  $f'_i(x),\ i=1,\dots,n,$ exist
 but $f$ is not Fr\' echet  differentiable.
Our aim is to prove that 
\begin{equation}\label{spoj}
\text{$A$ is a first category set if $f$ is continuous}
\end{equation}
and 
\begin{equation}\label{lips}
\text{$A$ is a $\sigma$-porous set if $f$ is Lipschitz.}
\end{equation}
Notice that \eqref{df} and Proposition \ref{borpart} give that $A$ is a Borel set (and hence a Souslin set) 
in $X$. Thus Theorem  \ref{richcuth} 
	implies that  there exists a rich family $\cal F_1$ 
of  separable subspaces of $X$ such that for every $V \in \cal F_1$ we have that
\begin{equation}\label{srcat}
	\text{$A$ is of the first category in $X$ whenever  $A\cap V$ is of the first category in $V$}
	\end{equation}
	 and
\begin{equation}\label{srpor}
	\text{$A$ is of $\sigma$-porous in $X$ whenever  $A\cap V$ is $\sigma$-porous in $V$.}
	\end{equation}

By Lemma \ref{souc}, the family
	$$ \F_2:=  \{V_1\times\cdots\times V_n:\ V_k \ \text{is a  separable subspace of}\ X_k,\ 1\leq k \leq n\} $$
is a rich family of  separable subspaces of $X$. 

By Theorem \ref{richfr}   there exists a rich family $\cal F_3$ of separable subspaces of $X$ such that for every $V \in \cal F_3$, $f$ is Fr\' echet differentiable (with respect to $X$) at every $x \in V \cap G$ at which its restriction to  $V\cap G$  is Fr\' echet differentiable (with respect to $V$).
	
Now, by Lemma \ref{richpr}, $\F:= \F_1 \cap \F_2 \cap \mathcal{F}_3$ is a rich family of separable subspaces 
of $X$.
	
	Choose an arbitrary $V \in \F$; since $V \in \F_2$, we have 
	$V= V_1\times\cdots\times V_n$ where $V_k$ is a  separable subspace of $X_k$,\ $1\leq k \leq n$. 
	
	If $A \cap V\neq \emptyset$, set
	$g:=  f|_{(V\cap G)}$.  Using the definition of $A $, we easily see that $g_i'(x)= g'_{V_i}(x)$, $i=1,\dots, n$,
	 exist for each $x \in A \cap V$. 
	Since  $f$ is  Fr\' echet nondifferentiable at each $x \in A \cap V$ and $V \in \F_3$, we conclude
	 that $g$ is   Fr\' echet nondifferentiable at each $x \in A \cap V$. Further,
	by the assumptions of the theorem,
	the space $\Li(\tilde X_i, \tilde Y)$ is separable whenever  $\tilde X_i$ is a separable subspace
	 of $V_i$, $i=1,\dots, n-1$, and $\tilde Y$ is a separable subspace of $Y$.
	So, using for $g$ on $G\cap V \subset V = V_1\times \dots \times V_n$ the special case of the theorem proved in the first step of the proof, we 
		 obtain that  $A \cap V$ is a first category (resp. $\sigma$-porous) set in $V$ if $f$ is
		 continuous (resp. Lipschitz). Therefore, since $V \in \F_1$, we obtain that both \eqref{srcat}
		 and \eqref{srpor} hold.
		\end{proof}
		
		Using \eqref{asp}, we see that  Theorem   \ref{sousepred} has the following interesting consequence.

\begin{corollary}\label{aspl}
	Let  $X_1,\dots,X_{n}$ be Banach spaces such that  $X_1,\dots,X_{n-1}$ are Asplund spaces, and let 
	$G \subset X:= X_1 \times \cdots \times X_n$ be  
	 an open set. 
		Let $f$ be a continuous (resp. Lipschitz) real function on $G$.
		
		Then there exists a first category (resp. $\sigma$-porous) set $A \subset G$ such that, for each $x \in G \setminus A$, the following implication holds.
	$$\text{$f$ has all Fr\' echet partial derivatives at $x$\ \ $\Rightarrow$\ \ $f$ is  Fr\' echet differentiable at 
$x$.} $$
\end{corollary}

\begin{remark}\label{soured}
\begin{enumerate}
\item[(i)]  There are also other concrete consequences of Theorem  \ref{sousepred}. For example,
 using  facts from Remark \ref{ljesep}, it is easy to show that Theorem  \ref{sousepred} can be used if each 
 $X_i$, $i=1,\dots, n-1$, is a Hilbert space (resp. a subspace of $c_0(\Gamma)$) and
  $Y= \ell_q(\Gamma)$, $1\leq q<2$, (resp.  $Y$ has the Radon-Nikod\' ym property).
	\item[ (ii)]  Theorem \ref{sousepred} can be further slightly generalized to a more complicated
	 general theorem (working with ``rich families of $\tilde X_i$'') which has further concrete applications.
	 \item[(iii)]  I do not know any example excluding the possibility that the assumptions concerning
	Banach spaces  $X_1$,\dots,$X_n$, $Y$ can be omitted in
	Theorem  \ref{sousepred}. From this reason the observations in (i) and
	 (ii) are mentioned without any details.
	\end{enumerate}
\end{remark}

By the separable reduction method, we prove also the following result on continuous functions whose all partial functions
 are DC (recall that a function on a Banach space is called DC if it is the difference of two continuous
 convex functions).

\begin{proposition}\label{pardc}
Let $X_1,\dots,X_n$ be Asplund spaces and $f$ a continuous real function on $X:=X_1\times\dots\times X_n$.
 Let each partial function
$$  f(x_1,\dots,x_{i-1},\cdot,x_{i+1},\dots,x_n),\ \ \ 1\leq i \leq n,$$
is DC on $X_i$. Then $f$ is generically Fr\' echet differentiable.
\end{proposition}
\begin{proof}
By Lemma \ref{souc}, the family
	$$ \F:=  \{V_1\times\cdots\times V_n:\ V_k \ \text{is a  separable subspace of}\ X_k,\ 1\leq k \leq n\} $$
is a rich family of  separable subspaces of $X$. Choose an arbitrary $V=V_1\times\cdots\times V_n \in \F$
 and set $g:= f|_V$. Then each partial function
$$  g(v_1,\dots,v_{i-1},\cdot,v_{i+1},\dots,v_n),\ \ \ 1\leq i \leq n,$$
 is DC on the Asplund space (see \eqref{asp}) $V_i$ and consequently is generically differentiable on $V_i$.
Since the spaces $V_1^*,\dots, V_{n-1}^*$  are separable, we can use Proposition \ref{gdpf} and obtain that
 $g$ is generically differentiable. Therefore $f$ is generically differentiable by Theorem \ref{red}.
 \end{proof}
\begin{remark}
\begin{enumerate}
\item
In the case of a   locally Lipschitz $f$,  Proposition \ref{pardc} immediately follows from
 \cite[Corollary 8.1]{Za12} (and also from \cite[Corollary 8.4]{Za12}).
\item
The proof of Proposition \ref{pardc} works if we weaken the assumption that each partial function of
 $f$ is DC to the assumption that it is a difference of two approximately convex functions. Thus
 we obtain that, in  \cite[Corollary 8.1]{Za12}, it is possible to suppose the continuity of $f$ 
 instead of the local Lipschitzness of $f$.
\item
Proposition \ref{pardc} could be proved quite similarly as Theorem \ref{sousepred}, but the present proof based on Theorem \ref{red} is shorter.
\end{enumerate}
\end{remark}
\bigskip

{\bf Acknowledgements.}\ \ \ I thank Ond\v rej Kalenda who suggested the generalization of Corollary 
 \ref{aspl}  to Theorem \ref{sousepred}.

\end{document}